\documentclass[twoside]{article} 
\usepackage{graphicx}
\usepackage{amsfonts}
\usepackage[fleqn]{amsmath}
\usepackage{amssymb}
\usepackage{amsthm}
\usepackage{amscd}
\usepackage[T1]{fontenc}
\usepackage{color}
\usepackage{afterpage}  %
\usepackage{fancyhdr}
\theoremstyle{plain}
\newtheorem{theorem}{Theorem}[section]
\newtheorem{lemma}[theorem]{Lemma}
\newtheorem{proposition}[theorem]{Proposition}

\newtheorem{corollary}[theorem]{Corollary}
\theoremstyle{definition}
\newtheorem{definition}[theorem]{Definition}
\newtheorem{example}[theorem]{Example}
\theoremstyle{remark}

\begin{document}

\afterpage{\rhead[]{\thepage} \chead[\small W. A. Dudek and R. A. R. Monzo \ \ \ \ ]
{\small Translatable semigroups \ \ \ \ \ \ \ } \lhead[\thepage]{} }                  

\begin{center}
\vspace*{2pt}
{\Large \textbf{Translatability and translatable semigroups}}\\[30pt]
 {\large \textsf{\emph{Wieslaw A. Dudek \ and \ Robert A. R. Monzo}}}\\[30pt]
\end{center}
 {\footnotesize\textbf{Abstract.} The concept of a $k$-translatable groupoid is explored in depth. Some properties of idempotent $k$-translatable groupoids, left cancellative $k$-translatable groupoids and left unitary $k$-translatable groupoids are proved. Necessary and sufficient conditions are found for a left cancellative $k$-translatable groupoid to be a semigroup.  Any such semigroup is proved to be left unitary and a union of $k$ disjoint copies of cyclic groups of the same order. Methods of constructing $k$-translatable semigroups that are not left cancellative are given.}

\section{Introduction}

This is the third in a series of four articles related to the fine algebraic structure of quadratical quasigroups, the general theory of which is still in its infancy. The main aim of these papers is to stimulate interest and further research in the algebraic theory of quadratical quasigroups. This theory has its geometrical implications \cite{Vol1}.

The focus of this paper is on translatable groupoids. In our previous paper \cite{DM2} it has been proved that the property of translatability can be hidden in the Cayley tables of certain finite quadratical quasigroups, including all those "induced" by $\mathbb{Z}_n$, the additive group of integers modulo $n$, where $n=4t+1$ for some positive $t$. Here we diverge from the study of quadratical quasigroups and explore some of the properties of translatable groupoids and translatable semigroups.

We begin in Section $2$ by listing three different conditions, each of which is necessary and sufficient for a finite groupoid to be translatable (Lemma \ref{L-basic}). Then we list necessary and sufficient conditions on a left cancellative translatable groupoid for it to be idempotent, elastic, medial, left or right distributive, commutative or associative (Lemma \ref{L-cond}). It is then proved that a groupoid has a $k$-translatable sequence, with respect to a particular ordering of its elements, if and only if it has a $k$-translatable sequence with respect to $n$ different orderings (Corollary \ref{C-order}). Essentially this result shows that, in a $k$-translatable groupoid, moving all the elements of the ordering of the Cayley table up one place (or down one place) preserves $k$-translatability. This result is applied later to prove that left cancellative $k$-translatable semigroups are left unitary (Theorem \ref{T-reord}) It is also proved that a left cancellative $k$-translatable groupoid of order $n$ is alterable if and only if $k^2$ is equivalent to $n-1$ modulo $n$ (Corollary \ref{C-alter}).

Idempotent $k$-translatable groupoids are considered next. We prove that such groupoids are left cancellative. Any two idempotent groupoids of the same order, each $k$-translatable for the same value of $k$, are proved to be isomorphic (Theorem \ref{izo}). An idempotent $k$-translatable groupoid of order $n$ is proved to be right distributive if and only if it is left distributive and, if the greatest common divisor of $k$ and $n$ is 1, it is a quasigroup Proposition \ref{P2} and Theorem \ref{T-quasi}).

Section $3$ explores left unitary translatable groupoids. Two such groupoids of the same order are isomorphic (Theorem \ref{T-izo}) and any such groupoid is medial but not right distributive (Theorem \ref{med}). The value of $k$ in a left unitary $k$-translatable groupoid determines whether the groupoid has various properties such as paramediality, elasticity left distributivity and right or left modularity (Theorem \ref{para}).

In Section $4$ we prove that a left cancellative $k$-translatable groupoid of order $n$ has a $k^*$-translatable dual groupoid if and only if $kk^*$ is equivalent to $1$ modulo $n$ (Theorem \ref{dual}). 
		
Section $5$ deals with $k$-translatable semigroups. In one of the main results of this paper, we find necessary and sufficient conditions on a left cancellative $k$-translatable groupoid for it to be a semigroup (Theorem \ref{semi}). Any such semigroup has a left neutral element (Corollary \ref{lmonoid}) and can be re-ordered so that it is a left unitary $k$-translatable semigroup (Theorem \ref{T-reord}). Two left cancellative $k$-translatable semigroups of the same order are isomorphic (Corollary \ref{C-izo}). A left cancellative $k$-translatable semigroup is a disjoint union of cyclic groups of the same order (Theorem \ref{decomp}). A left cancellative $(n-1)$-translatable semigroup of order $n$ and a paramedial left cancellative $k$-translatable semigroup of order $n$ are isomorphic to the additive group of integers modulo $n$. The section ends with various results that allow us to calculate $k$-translatable sequences that yield semigroups that are not left cancellative (Theorem \ref{Tsem} to Corollary \ref{nsemi}).

In Section $6$ we find two different conditions under which $t$ disjoint copies of a left unitary $k$-translatable semigroup $Q$ of order $n$ can be embedded in a left unitary translatable semigroup $G$ of order $tn$, such that $G$ is a disjoint union of $t$ copies of $Q$. In both cases $k=tq$. If $(k+k^2)\equiv 0({\rm mod}\,tn)$ then $G$ is $k$-translatable (Theorem \ref{T-62}). If $n=k+k^2$ then $G$ is $(k+(t-1)n)$-translatable (Theorem \ref{T-63}).

\section*{\centerline{2. Translatable groupoids}}\setcounter{section}{2}\setcounter{theorem}{0}

For simplicity we assume that all groupoids considered in this note are finite and have form $Q=\{1,2,\ldots,n\}$ with the {\it natural ordering} $1,2,\ldots,n$, which is always possible by renumeration of elements. Moreover, instead of $i\equiv j({\rm mod}\,n)$ we will write $[i]_n=[j]_n$. Additionally, in calculations of modulo $n$, we assume that $0=n$. 

\begin{definition} A finite groupoid $Q$ is called {\it $k$-translatable}, where $1\leqslant k< n$, if its multiplication table is obtained by the following rule: If the first row of the multiplication table is $a_1,a_2,\ldots,a_n$, then the $q$-th row is obtained from the $(q-1)$-st row by taking the last $k$ entries in the $(q-1)-$st row and inserting them as the first $k$ entries of the $q$-th row and by taking the first $n-k$ entries of the $(q-1)$-st row and inserting them as the last $n-k$ entries of the $q$-th row, where $q\in\{2,3,\ldots,n\}$. Then the (ordered) sequence $a_1,a_2,\ldots,a_n$ is called a {\it $k$-translatable sequence} of $Q$ with respect to the ordering $1,2,\ldots,n$. A groupoid is called  {\it translatable} if it has a $k$-translatable sequence for some $k\in\{1,2,\ldots,n-1\}$. 
\end{definition}
 							
It is important to note that a $k$-translatable sequence of a groupoid $Q$ depends on the ordering of the elements in the multiplication table of $Q$. A groupoid may be $k$-translatable for one ordering but not for another.

\begin{example}\label{Ex1} Consider the following groupoids:
 {\small
$$
\begin{array}{lcccr}
\arraycolsep=1.5mm   \arraycolsep=1.2mm
\begin{array}{c|cccccc}						
\cdot&1&	2&	3&	4\\ \hline
1&	1&	2&	3& 4\\
2&	2&	3&	4&	1\\
3&	3&	4&	1&	2\\
4&	4&	1&	2&	3
\end{array}&
\ \ \ \ \ & \arraycolsep=1.2mm
\begin{array}{c|cccccc}
\cdot&	1&	3&	4&	2\\ \hline
1&	1&	3&	4&	2\\
3&	3&	1&	2&	4\\
4&	4&	2&	3&	1\\
2&	2&	4&	1&	3
\end{array}&
\ \ \ \ \ &\arraycolsep=1.2mm
\begin{array}{c|ccccc}
\cdot&1&2&3&4\\ \hline
1&1&4&3&2\\
2&3&2&1&4\\
3&1&4&3&2\\
4&3&2&1&4
\end{array}
\end{array}
$$}

The first table determines a $3$-translatable semigroup isomorphic to the additive group $\mathbb{Z}_4$. The second table says that after change of ordering this semigroup is not translatable. The last table determines an idempotent $2$-translatable groupoid which is not a semigroup.
\end{example}

On the other hand, a groupoid with the operation $x\cdot y=a$, where $a$ is fixed, is a $k$-translatable semigroup for each $k$, but the semigroup $(Q,\cdot)$ with the operation  
$$
x\cdot y=\left\{\begin{array}{llll} 1& {\rm if}\ x+y \ {\rm is \ even},\\ 
2& {\rm if}\ x+y \ {\rm is\ odd}
\end{array}\right.
$$
is $k$-translatable for each odd $k<n$, if $n$ is even, and $k$-translatable for each even $k$, if $n$ is odd. All these groupoids are easily proved to be semigroups.

\begin{lemma}
For a fixed ordering a left cancellative groupoid may be $k$-transla\-ta\-ble for only one value of $k$.
\end{lemma} 
\begin{proof}
By renumeration of elements we can assume that a groupoid $Q$ of order $n$ has a natural ordering $1,2,\ldots,n$.
If the first row of the multiplication table of such groupoid has the form $a_1,a_2,\ldots,a_n$, then all $a_i$ are different. The first element of the second row is equal to $a_{n-k+1}$ if a groupoid is $k$-translatable, or $a_{n-t+1}$ if it is $t$-translatable. So, $1\cdot (n-k+1)=a_{n-k+1}=2\cdot 1=a_{n-t+1}=1\cdot (n-t+1)$, which, by left cancellativity, implies $n-k+1=n-t+1$, and consequently $k=t$. 
\end{proof}
Obviously, for another ordering it may be $k$-translatable for some other value of $k$, but in some cases the change of ordering preserves $k$-translatability.

\medskip
We start with the following simple observation.
\begin{lemma}\label{L-lc}
A $k$-translatable groupoid containing at least one left cancellable element is left cancellative.
\end{lemma}
\begin{proof}
Indeed, let $i$ be the left cancellable element of a $k$-translatable groupoid $Q$. Then, $i$-th row of the multiplication table of this groupoid contains different elements. By $k$-translability other rows also contain different elements. So, for all $j,s,t\in Q$ from $j\cdot s=j\cdot t$ it follows that $s=t$. Hence, $Q$ is left cancellative.
\end{proof}

The proofs of the following two Lemmas are straightforward and are omitted.

\begin{lemma}\label{L-basic}
Let $a_1,a_2,\ldots,a_n$ be the first row of the multiplication table of a groupoid $Q$ of order $n$. Then $Q$ is $k$-translatable if and only if for all $i,j\in Q$ the following $($equivalent$)$ conditions are satisfied:
\begin{enumerate}
\item[$(i)$] \ $i\cdot j=a_{[(i-1)(n-k)+j]_n}=a_{[k-ki+j]_n}$,
\item[$(ii)$] \ $i\cdot j=[i+1]_n\cdot [j+k]_n$,
\item[$(iii)$] \ $i\cdot [j-k]_n=[i+1]_n\cdot j$.
\end{enumerate} 
\end{lemma}
Recall that a groupoid $Q$ is called

$\bullet$ {\em alterable} if $i\cdot j=w\cdot z$ implies $j\cdot w=z\cdot i$,

$\bullet$ {\em bookend} if and only if $(j\cdot i)\cdot (i\cdot j)=i$, 

$\bullet$ {\em paramedial} if and only if $(i\cdot j)\cdot (w\cdot z)=(z\cdot j)\cdot (w\cdot i)$,

$\bullet$ {\em strongly elastic} if and only if $i\cdot (j\cdot i)=(i\cdot j)\cdot i=(j\cdot i)\cdot j$,

$\bullet$ {\em left modular} if and only if $(i\cdot j)\cdot z=(z\cdot j)\cdot i$,

$\bullet$ {\em right modular} if and only if $i\cdot (j\cdot z)=z\cdot (j\cdot i)$\\
for all $i,j,w,z \in Q$.

\begin{lemma}\label{L-cond}
For a $k$-translatable left cancellative groupoid $Q$ of order $n$ the following statements are valid:
\begin{enumerate}
\item[$(i)$] \ $a_i=a_j$ if and only if $i=j$,
\item[$(ii)$] \ $Q$ is idempotent if and only if $i=a_{[k-ki+i]_n}$ for all $i\in Q$, 
\item[$(iii)$] \ if $Q$ is idempotent then $i\cdot j=t$ if and only if $[j-ki]_n=[t-kt]_n$ for all $i,j,t\in Q$, 
\item[$(iv)$] \ $Q$ is elastic if and only if $[i+ki]_n=[(j\cdot i)+k(i\cdot j)]_n$ for all $i,j\in Q$,
\item[$(v)$] \ $Q$ is strongly elastic if and only if \ $[i+ki]_n=[(j\cdot i)+k(i\cdot j)]_n$ and $[i+kj]_n=[(i\cdot j)+k(i\cdot j)]_n$ for all $i,j\in Q$,
\item[$(vi)$] \ $Q$ is bookend if and only if $i=a_{[k-k(j\cdot i)+(i\cdot j)]_n}$ for all $i,j\in Q$, 
\item[$(vii)$]v \ $Q$ is left distributive if and only if $[(i\cdot j)+k(s\cdot i)]_n=[(s\cdot j)+ks]_n$ for all $i,j,s\in Q$, 
\item[$(viii)$] \ $Q$ is right distributive if and only if $[s+k(i\cdot s)]_n=[(j\cdot s)+k(i\cdot j)]_n$ for all $i,j,s\in Q$, 
\item[$(ix)$] \ $Q$ is medial if and only if $[(w\cdot z)+k(i\cdot w)]_n=[(j\cdot z)+k(i\cdot j)]_n$ for all $i,j,w,z\in Q$, 
\item[$(x)$] \ $Q$ is alterable if and only if $[j+kw]_n=[z+ki]_n$ implies $[w+kz]_n=[i+kj]_n$ for all $i,j,w,z\in Q$,
\item[$(xi)$] \ $Q$ is commutative if and only if $k=n-1$,
\item[$(xii)$] \ $Q$ is associative if and only if $[i+kj]_n=[(s\cdot i)+k(j\cdot s)]_n$ for all $i,j,s\in Q$.
\end{enumerate}
\end{lemma}

\begin{lemma}\label{L-3} The sequence $a_1,a_2,\ldots,a_n$  is a $k$-translatable sequence of $Q$  with respect to the ordering $1,2,\ldots,n$  if and only if $a_k,a_{k+1},\ldots,a_n,a_1,\ldots,a_{k-1}$ is a $k$-translatable sequence of $Q$  with respect to the ordering $n,1,2,\ldots,n-1$. 
\end{lemma}
\begin{proof} $(\Rightarrow)$:
The new ordering can be written as $1',2',\ldots,n'$, where $i'=[i-1]_n$. Then, since $Q$ is $k$-translatable with respect to the ordering $1,2,\ldots,n$, by Lemma \ref{L-basic}$(ii)$, $[i'+1]_n\cdot[j'+k]_n=i\cdot[j-1+k]_n=[i-1]_n\cdot [j-1]_n=i'\cdot j'$.
Hence, $Q$ is $k$-translatable with respect to the ordering $n,1,2,\ldots,n-1$, with $k$-translatable sequence $n\cdot n, n\cdot 1,n\cdot 2,\ldots,n\cdot (n-1)$. That is, using Lemma \ref{L-basic}$(i)$, the $k$-translatable sequence is $a_k,a_{k+1},\ldots,a_n,a_1, a_2,\ldots,a_{k-1}$.

 $(\Leftarrow)$: The order $n,1,2,\ldots,n-1$ can be written as $1',2',\ldots,(n-1)',n'$, where $i'=[i-1]_n$ for all $i\in\{1,2,\ldots,n\}$. The order $1,2,\ldots,n$ can be written as $i=[i'+1]_n$. Then, by Lemma \ref{L-basic}$(ii)$, we have $[i+1]_n\cdot [j+k]_n=[i+2]'_n\cdot [j+k+1]'_n=[i+1]'_n\cdot [j+1]'_n=i\cdot j$ and so $1\cdot 1,1\cdot 2,\ldots,1\cdot n$ is a $k$-translatable sequence with respect to the order $1,2,\ldots,n$. But $a_k,a_{k+1},\ldots,a_n,a_1,\ldots,a_{k-1}$ is the $k$-translatable sequence with respect to the order $1',2',\ldots,n'$. Let's call this sequence $b_1,b_2,\ldots,b_n$. Note, then, that $b_i=a_{[i-1+k]_n}$. So we have $1\cdot j=[i+1]'_n\cdot [j+1]'_n=2'\cdot [j+1]'_n=1'\cdot [j+1-k]'_n=b_{[j+1-k]_n}=a_j$. This implies that the sequence $a_1,a_2,\ldots,a_n$ is a $k$-translatable sequence with respect to the order $1,2,\ldots,n$.  
\end{proof}

\begin{corollary}\label{C-order}
A $k$-translatable groupoid of order $n$ has at least $n$ \ $k$-translatable sequences, each with respect to a different order.
\end{corollary}

A groupoid $Q$ is called {\it right solvable $($left solvable$)$} if for any $\{a,b\} \subseteq Q$ there exists a unique $x\in G$ such that $ax = b$ ($xa = b$). It is a {\it quasigroup} if it is left and right solvable. 

\begin{proposition}\label{P3}
An alterable, right solvable and right distributive groupoid is an idempotent quasigroup.
\end{proposition}
\begin{proof}
A right solvable groupoid $Q$ is left cancellative. Right distributivity and left cancellativity give idempotency.

Then the right solvability means that for all $z,w\in Q$ there are uniquely determined elements $q,q'\in Q$ such that $z\cdot q=w$ and $w\cdot q'=z\cdot w$. Then, using alterability, $w=w\cdot w=q'\cdot z$. The element $q'$ is uniquely determined. Indeed, if $w=q'\cdot z=q''\cdot z$, then, by alterability, $z\cdot q''=z\cdot q'$. This implies $q'=q''$.
Hence $Q$ is also left solvable and, hence, it is a quasigroup.
\end{proof}

\begin{lemma}\label{L-idtrans}
For fixed $k\ne 1$ and $n$ there is an idempotent $k$-translatable groupoid of order $n$. It is left cancellative. 
\end{lemma}
\begin{proof}
First observe that an idempotent groupoid cannot be $k$-translatable for $k=1$. Let $k>1$. Then, by Lemma \ref{L-basic}, in an idempotent $k$-translatable groupoid $Q$, for all $i\in Q$ we have $i=a_{[k-ki+i]_n}$. Since $Q$ is idempotent, the first row of the multiplication table of this groupoid contains the distinct elements $a_1,a_{[2-k]_n},a_{[3-2k]_n},\ldots,a_{[n+k(1-n)]_n}$, with $i$ appearing in the $[k-ki+i]_n$-column. So, by Lemma \ref{L-lc}, this groupoid is left cancellative. By $k$-translatability,  other rows are uniquely determined by the first. 
\end{proof}

An idempotent $k$-translatable groupoid may not be right cancellative. It is not difficult to see that 
if the first row of the multiplication table of the $4$-translatable  groupoid $Q$ of order $8$ has form $1,6,3,8,5,2,7,4$, then this groupoid is idempotent and left cancellative but it is not right cancellative.


For every odd $n$ and every $k>1$ such that $(k,n)=1$ there is at least one idempotent $k$-translatable quasigroup. 
For even $n$ there are no such quasigroups \cite[Theorem 8.9]{DM2}.

\begin{lemma}
An idempotent $k$-translatable groupoid that is alterable, strongly elastic or bookend is cancellative, i.e., it is a quasigroup.
\end{lemma}
\begin{proof}
By Lemma \ref{L-idtrans} such groupoid is left cancellative. Let $i\cdot j=s\cdot j$. Then obviously $s=[i+t]_n$ for some $t\in\{1,2,\ldots,n\}$. Since $a_{[k-ki+j]_n}=i\cdot j=s\cdot j=a_{[k-ks+j]_n}$, $[k-ki+j]_n=[k-ks+j]_n=[k-ki-kt+j]_n$ and so $[kt]_n=0$. But then, since $Q$ is idempotent, $t\cdot t=t=a_{[k-kt+t]_n}=a_{[k+t]_n}=1\cdot[k+t]_n=n\cdot t$. If $Q$ is alterable, by definition, $i\cdot j= w\cdot z$ implies $j\cdot w=z\cdot i$. Therefore, $t\cdot n=t\cdot t=t$. If $Q$ is strongly elastic then, by definition, $(i\cdot j)\cdot i=(j\cdot i)\cdot j$. Therefore $t=n\cdot t=t\cdot t=t\cdot (n\cdot t)=(n\cdot t)\cdot n=t\cdot n$. In either case, left cancellation implies $t=n$. If $Q$ is bookend then, by definition, $(i\cdot j)\cdot (j\cdot i)=j$. So, $t=t\cdot t=n\cdot t=(n\cdot t)\cdot (t\cdot n)=t\cdot (t\cdot n)$ and left cancellation implies $t=n$. Hence, if $Q$ is alterable, strongly elastic or bookend, $t=n$ and so $s=[i+t]_n=[i+n]_n=i$. Therefore, $Q$ is also right cancellative, and consequently, it is a quasigroup.  
\end{proof}

\begin{theorem}\label{izo}
Idempotent $k$-translatable groupoids of the same order are isomorphic.
\end{theorem}
\begin{proof}
Suppose that $(Q,\cdot)$ and $(S,*)$ are idempotent $k$-translatable groupoids of order $n$. Then $Q=\{1,2,\ldots,n\}$ has a $k$-translatable sequence $a_1,a_2,\ldots,a_n$ with respect to the ordering $1,2,\ldots,n$, and $S=\{c_1,c_2,\ldots,c_n\}$ has a $k$-translatable sequence $b_1,b_2,\ldots,b_n$ with respect to the ordering $c_1,c_2,\ldots,c_n$. Then, by Lemma \ref{L-basic} $(i)$, we have $i\cdot j=a_{[k-ki+j]_n}$ and $c_i*c_j=b_{[k-ki+j]_n}$. Since both groupoids are idempotent, $i=a_{[k-ki+i]_n}$ and $c_i=b_{[k-ki+i]_n}$ for all elements of $Q$ and $S$.

Define a mapping $\varphi\colon Q\longrightarrow S$ as follows: $\varphi(i)=c_i$ (for all $i\in Q$). Then $\varphi$ is clearly a bijection. For $i,j\in Q$, we have $t=i\cdot j=a_{[k-ki+j]_n}=a_{[k-kt+t]_n}$, and since $Q$ is left cancellative (Lemma \ref{L-idtrans}), by Lemma \ref{L-cond}, we obtain $[k-ki+j]_n=[k-kt+t]_n$. Thus $\varphi(i\cdot j)=c_{i\cdot j}=c_t=b_{[k-kt+t]_n}=b_{[k-ki+j]_n}=c_i *c_j=\varphi(i)*\varphi(j)$ and so $\varphi$ is an isomorphism. 
\end{proof}

\begin{proposition}\label{P2}
An idempotent, $k$-translatable groupoid $Q$ of order $n$ is right distributive if and only if it is left distributive.
\end{proposition}
\begin{proof}
Let $i\cdot j=t$, $i\cdot s=u$, $j\cdot s=w$. The right distributivity means that $t\cdot s=u\cdot w$. Then, by Lemmas \ref{L-idtrans} and \ref{L-cond}, \ $[s+ku]_n=[w+kt]_n$, \ $[t+ki]_n=[j+kt]_n$, \ $[u+ki]_n=[s+ku]_n$ \ and \ $[s+kw]_n=[w+kj]_n$. So, $[s+ku]_n=[w+kt]_n=[u+ki]_n$, and consequently, $[-ki+w]_n=[-kt+u]_n$. Thus $i\cdot w=t\cdot u$. Hence, $Q$ is left distributive.

The converse statement can be proved analogously.
\end{proof}

\begin{theorem}\label{T-quasi}
An idempotent $k$-translatable groupoid of order $n$ with $(k,n)=1$ is a quasigroup.
\end{theorem}
\begin{proof}
By Lemma \ref{L-idtrans} such groupoid is left cancellative. To prove that it is right cancellative consider an arbitrary column of its multiplication table. Suppose that it is $j$-th column. This column has form $a_j,a_{[j-k]_n},a_{[j-2k]_n},\ldots,a_{[j-(n-1)k]_n}$. All these elements are different. In fact, if $a_{[j-sk]_n}=a_{[j-tk]_n}$ for some $s,t\in Q$, then $(s-t)k\equiv 0({\rm mod}\,n)$, which is possible only for $s=t$ because $(k,n)=1$. So, in this groupoid $s\cdot j=t\cdot j$ implies $s=t$ for all $j,s,t\in Q$. Hence $Q$ is a quasigroup.
\end{proof}

\begin{lemma}
If a $k$-translatable groupoid of order $n$ has a right cancellable element, then $(k,n)=1$.
\end{lemma}
\begin{proof}
Let $t\in Q$ be a right cancellable element of a $k$-translatable groupoid $Q$. Then the mapping $R_t(x)=x\cdot t$ is one-to-one. So, it is a bijection. Hence all elements of the $t$-column are different. Thus $a_{[k-ki+t]_n}=i\cdot t=j\cdot t=a_{[k-kj+t]_n}$ is equivalent to $k(i-j)\equiv 0({\rm mod}\,n)$. The last, for all $i,j\in Q$, gives $i=j$ only in the case when $(k,n)=1$.
\end{proof}

\begin{proposition}\label{P1}
An idempotent, $k$-translatable groupoid $Q$ of order $n$ is elastic if and only if $[(i\cdot j)+(j\cdot i)]_n=[i+j]_n$ holds for all $i,j\in Q$.
\end{proposition}
\begin{proof}
Since $Q$ is idempotent, for all $i,j\in Q$ we have $i\cdot j=(i\cdot j)\cdot (i\cdot j)$, i.e., 
\begin{equation}\label{eid}
[-k(i\cdot j)+(i\cdot j)]_n=[-ki+j]_n,
\end{equation}
which together with elasticity and Lemma \ref{L-cond}$(iv)$, gives $[j+k(i\cdot j)-(i\cdot j)]_n=[ki]_n=[(j\cdot i)+k(i\cdot j))-i]_n$. This implies $[(i\cdot j)+(j\cdot i)]_n=[i+j]_n$. 

Conversely, if $[(i\cdot j)+(j\cdot i)]_n=[i+j]_n$, then also $[k(i\cdot j)+k(j\cdot i)]_n=[ki+kj]_n$, and consequently 
$$[i+ki]_n=[i+k(i\cdot j)+k(j\cdot i)-kj]_n=[(k(j\cdot i)-kj+i)+k(i\cdot j)]_n\stackrel{\eqref{eid}}{=}[(j\cdot i)+k(i\cdot j)]_n.
$$
This gives the elasticity of $Q$.
\end{proof}

\begin{theorem}\label{T2}
An alterable, left or right cancellative groupoid of order $n$ may be $k$-translatable only for $[k^2]_n=n-1$.
\end{theorem}
\begin{proof}
Since in alterable groupoids left and right cancellativity are equivalent, we will consider only an alterable groupoid with left cancellativity. 

Let $Q$ be an alterable groupoid which is left cancellative and $k$-translatable. Then, by Lemma \ref{L-basic}$(iii)$, for all $i,s\in\{1,2,\ldots,n\}$ we have 
\begin{equation}\label{e11}
\rule{20mm}{0mm}[i+1]_n\cdot s=i\cdot [n+s-k]_n=i\cdot [s-k]_n
\end{equation}
for all $i,s\in\{1,2,\ldots,n\}$.

Thus, 
$$
(n-1)\cdot n=(n-1)\cdot [n+k-k]_n\stackrel{\eqref{e11}}{=}n\cdot [n+k]_n=n\cdot k,
$$
which, by alterability, gives
\begin{equation}\label{e12}
\rule{40mm}{0mm}n\cdot n=k\cdot(n-1).
\end{equation}
But $n\cdot n=n\cdot [n+k-k]_n\stackrel{\eqref{e11}}{=}[n+1]_n\cdot [n+k]_n=[n+1]_n\cdot k$, so
$$
[n+1]_n\cdot k=k\cdot (n-1),
$$
which, by alterability, implies
$k\cdot k=[n-1]_n\cdot [n+1]_n$. Hence
$$
k\cdot k=(n-1)\cdot [n+1]_n\stackrel{\eqref{e11}}{=}n\cdot [n+1+k]_n=n\cdot (k+1)\stackrel{\eqref{e11}}{=}k\cdot [k+1+k^2]_n.
$$
Since $Q$ is left cancellable, we obtain $0=[1+k^2]_n$. Therefore $k^2\equiv (-1)({\rm mod}\,n)$.
\end{proof}

\begin{corollary}\label{C-alter}
A $k$-translatable and left cancellative groupoid of order $n$ is alterable if and only if  $[k^2]_n=n-1$. 
\end{corollary}
\begin{proof} Let a left cancellative groupoid $Q$ of order $n$ be $k$-translatable. Assume that the first row of the multiplication table of this groupoid has form $a_1,a_2,\ldots,a_n$. Then all these elements are different. 

If $k^2\equiv -1({\rm mod}\,n)$, then $i\cdot j=g\cdot h$, by Lemma \ref{L-basic}$(i)$, implies $[k(g-i)]_n=[h-j]_n$. Multiplying both sides of this equation by $k$ gives $[k^2(g-i)]_n=[k(h-j)]_n$, i.e., $[i-g]_n=[k(h-j)]_n$. This implies $j\cdot g=h\cdot i$. Therefore $Q$ is alterable.
	
	The converse statement is a consequence of Theorem \ref{T2}.
\end{proof}

\section*{\centerline{3. Left unitary translatable groupoids}}\setcounter{section}{3}\setcounter{theorem}{0}

By a {\it $($left$)$ unitary groupoid} of order $n$ we mean a groupoid $(Q,\cdot)$, where the set $Q=\{1,2,\ldots,n\}$ is naturally ordered and $1$ is its (left) neutral element.

The first row of a multiplication table of such groupoid has the form $1,2,\ldots,n$. So, if it is $k$-translatable, then, by Lemma \ref{L-basic}, $i\cdot j=[k-ki+j]_n$ for all $i,j\in Q$. Moreover, by Lemma \ref{L-lc}, a $k$-translatable left unitary groupoid is left cancellative. If $n$ is prime, then it is also right cancellative.

\begin{theorem}\label{T-izo} 
Left unitary $k$-translatable groupoids of the same order are isomorphic.
\end{theorem}
\begin{proof} Suppose that $(Q,\cdot)$ and $(G,*)$ are left unitary, $k$-translatable groupoids having $k$-translatable sequences $1,2,\ldots,n$ and $1',2',\ldots,n'$, with respect to the orderings $1,2,\ldots,n$ and $1',2',\ldots,n'$, respectively. Then $i\cdot j=[k-ki+j]_n$ and $i'*j'=[k-ki+j]'_n$. Define $\varphi\colon Q\longrightarrow G$ as $\varphi(i)=i'$. Then $\varphi$ is clearly a bijection and $\varphi(i\cdot j)=\varphi([k-ki+j]_n)=[k-ki+j]'_n=i'*j'=\varphi(i)*\varphi(j)$, so $\varphi$ is an isomorphism.
\end{proof}

\begin{theorem}\label{med}
A $k$-translatable left unitary groupoid is medial, but it is not right distributive.
\end{theorem}
\begin{proof}
In a $k$-translatable left unitary groupoid $Q$ of order $n$ we have  
$(x\cdot y)\cdot(z\cdot w)=(x\cdot z)\cdot(y\cdot w)\Leftrightarrow [k-k(kx+y)+(k-kz+w)]_n=
[k-k(k-kx+z)+(k-ky+w)]_n \Leftrightarrow [k(k-kx)]_n=[k(k-kx)]_n$. The last identity is always valid. So, this groupoid is medial.

If it is right distributive, then $(x\cdot y)\cdot z=(x\cdot z)\cdot(y\cdot z)$ for all $x,y,z\in Q$, or, equivalently, $[k-k(k-kx+y)+z]_n=[k-k(k-kx+z)+(k-ky+z)]_n$. This implies $[k(z-1)]_n=0$ for all $z\in Q$. Thus $k=0$, a contradiction. Hence, $Q$ cannot be right distributive.
\end{proof}

\begin{lemma}
A unitary groupoid of order $n$ may be $k$-trans\-latable only for $k=n-1$.
\end{lemma}
\begin{proof}
Indeed,  $i=i\cdot 1=[k-ki+1]_n$ for every $i\in Q$, which for $i=2$ gives $2=[1-k]_n$. This is possible only for $k=n-1$.
\end{proof}
\begin{corollary}\label{C-cyc}
Groups of order $n$ may be $k$-translatable only for $k=n-1$. Such groups are cyclic.
\end{corollary}

Obviously, such groups are $(n-1)$-translatable with respect to the ordering $1,2,\ldots,n$ or $n,1,2,\ldots,n-1$ (Lemma \ref{L-3}). Example \ref{Ex1} shows that in another order they may not be translatable.

\begin{theorem}\label{para}
If a $k$-translatable groupoid $Q$ of order $n$ is left unitary then $Q$ is 
\begin{enumerate}
\item[$(a)$] bookend if and only if $[k^2]_n=[2k]_n=n-1$,
\item[$(b)$] elastic if and only if $[k^2+k]_n=0$,
\item[$(c)$] left distributive if and only if $[k^2]_n=0$,
\item[$(d)$] left modular if and only if $k=n-1$,
\item[$(e)$] right modular if and only if $[k^2]_n=1$,
\item[$(f)$] paramedial if and only if $[k^2]_n=1$.
\end{enumerate}
\end{theorem}
\begin{proof}
$(a)$: Indeed, $i=(j\cdot i)\cdot (i\cdot j)\Leftrightarrow i=[k-k[k-kj+i]_n+[k-ki+j]_n]_n\Leftrightarrow [2k-k^2+(k^2+1)j-(2k+1)i]_n=0\Leftrightarrow [k^2+1]_n=[2k+1]_n=0$.

\smallskip
$(b)$: $i\cdot(j\cdot i)=(i\cdot j)\cdot i\Leftrightarrow [k-ki+(k-kj+i)]_n=[k-k(k-ki+j)+i]_n\Leftrightarrow [(k^2+k)i]_n=[k^2+k]_n\Leftrightarrow [k^2+k]_n=0$.

\smallskip
In other cases the proof is similar.
\end{proof}

\begin{corollary}
Let $Q$ be a left unitary $k$-translatable groupoid. Then
\begin{enumerate}
\item[$(a)$] $Q$ is right modular if and only if it is paramedial,
\item[$(b)$] if $Q$ left modular, then it is right modular, elastic and paramedial, 
\item[$(c)$] $Q$ is not strongly elastic.
\end{enumerate}
\end{corollary}

\begin{theorem}\label{embed} A $k$-translatable groupoid $Q$ of order $n$ is embeddable in a $k$-translatable groupoid $B$ of order $(t+1)n$ $(t\in\{1,2,3,\ldots\})$ such that
\begin{enumerate}
\item[$(a)$] if $Q$ is left cancellative then $B$ is left cancellative, and
\item[$(b)$] if $Q$ is left unitary then $B$ is left unitary.
\end{enumerate}
\end{theorem}
\begin{proof} Let $Q=\{1,2,\ldots,n\}$ have the $k$-translatable sequence $a_1,a_2,\ldots,a_n$. Define $B$ as follows: $B=\{b_{ij}:1\leq i\leq n, \ 1\leq j\leq t+1\}$, where $b_{i1}=i$ for all $i=1,2,\ldots,n$ and $B$ has $(t+1)n$ distinct elements. Consider $B$ with the following ordering: $1',2',\ldots,((t+1)n)'$, where $b_{ij}=((i-1)(t+1)+j)'$. 

Let the first row $c_1,c_2,\ldots,c_{(t+1)n}$ of the multiplication table of $B$ have the form 
$$
a_1,b_{12},\ldots,b_{1(t+1)},a_2,b_{22},\ldots,b_{2(t+1)},a_3,b_{32},\ldots,b_{3(t+1)},\ldots,a_n,b_{n2},\ldots,b_{n(t+1)}.
$$
Then $b_{ij}=c_{(i-1)(t+1)+j}$ for all $j\ne 1$ and $a_i=c_{(i-1)(t+1)+1}$. 

The second row is obtained from the first by taking the last $k$ entries and placing them as the first $k$ entries of the second row and taking the first $((t+1)n-k)$ entries of the first row and placing them as the last $((t+1)n-k)$ entries of the second. In the same way we construct the $(i+1)$-th row from the $i$-th (for all $i=2,3,\ldots,(t+1)n-1$). Clearly, such defined groupoid $B$ is $k$-translatable. Also, if $Q$ is left cancellative then $B$ is left cancellative and if $Q$ is left unitary then $B$ is left unitary. 

So, we need only prove that the product $\cdot$ in $Q$ is preserved in the new product, call it $*$, in $B$. In proving this, we do so using the facts that $Q$ and $B$ are $k$-trans\-la\-table.

Since 
$$
i=b_{i1}=((i-1)(t+1)+1)'=(it+i-t)'
$$ 
for all $i\in Q$, for any $i,j\in Q$, by Lemma \ref{L-basic}, we have
$$
i*j=(it+i-t)'*(jt+j-t)'=c_x,
$$
where $x=[k-k(it+i-t)+jt+j-t]_{(t+1)n}=[k(t+1)-ki(t+1)+j(t+1)-t]_{(t+1)n}=([k-ki+j]_n-1)(t+1)+1$.

But, on the other hand, using the fact that $a_i=c_{(i-1)(t+1)+1}$,
$$
i\cdot j=a_{[k-ki+j]_n}=c_{([k-ki+j]_n-1)(t+1)+1}=c_x,
$$ 
Therefore, $i\cdot j=i*j$ for all $i,j\in Q$.    
\end{proof}

\section*{\centerline{4. Dual groupoids}}\setcounter{section}{4}\setcounter{theorem}{0}

The {\em dual groupoid} of a groupoid $(Q,\cdot)$ is defined as a groupoid $(Q,*)$ with the operation $i*j=j\cdot i$. We asume $(Q,\cdot)$ and $(Q,*)$ both have the natural ordering $1,2,\ldots,n$.

\begin{theorem}\label{dual}
If a left cancellative groupoid $Q$ of order $n$ is $k$-translatable, then its dual groupoid $(Q,*)$ is $k^*$-translatable if and only if \ $[kk^*]_n=1$. 
\end{theorem}
\begin{proof}
$(\Rightarrow)$: Suppose that $(Q,\cdot)$, where $Q=\{1,2,\ldots,n\}$, is $k$-translatable and the first row of its multiplication table has the form $a_1,a_2,\ldots,a_n$. If its dual groupoid $(Q,*)$ is $k^*$-translatable, then $1*[1-k^*]_n=[1-k^*]_n\cdot 1=a_{[k-k(1-k^*)+1]_n}=a_{[kk^*+1]_n}$. But on the other hand, $1*[1-k^*]_n=2*1=1\cdot 2=a_2$. So, $a_{[kk^*+1]_n}=a_2$. Left cancellability implies $[kk^*+1]_n=2$, which gives $[kk^*]_n=1$.

$(\Leftarrow)$: Suppose that $[kk^*]_n=1$. Then $[i+1]_n*[j+k^*]_n=a_{[k-k(j+k^*)+(i+1)]_n}=a_{[k-kj-kk^*+i+1]_n}=a_{[k-kj+i]_n}=i*j.
$ Hence, $(Q,*)$ is $k^*$-translatable with respect to the ordering $1,2,\ldots,n$.
\end{proof}

\begin{corollary}
A left cancellative, $k$-trans\-la\-tab\-le groupoid $(Q,\cdot)$ of order $n$ and its dual groupoid $(Q,*)$ are $k$-translatable, for the same value of $k$, if and only if \ $[k^2]_n=1$.
\end{corollary}

\begin{corollary}
A left unitary, $k$-trans\-la\-tab\-le groupoid $(Q,\cdot)$ of order $n$ and its dual groupoid $(Q,*)$ are $k$-translatable, for the same value of $k$, if and only if $(Q,\cdot)$ is paramedial.
\end{corollary}

\begin{corollary}
A dual groupoid of a left cancellative, $k$-translatable groupoid $(Q,\cdot)$ of order $n$ is  $[n-tk]_n$-translatable if and only if \ $[tk^2]_n=n-1$.
\end{corollary}

\begin{corollary}\label{C45} 
A left cancellative $k$-translatable groupoid of order $n$ is alterable if and only if its dual groupoid is $(n-k)$-translatable.
\end{corollary}
\begin{proof} $(\Rightarrow)$: Since, by Corollary \ref{C-alter}, $[k^2]_n=n-1$ we have $[k(n-k)]_n=[-k^2]_n=1$. By Theorem \ref{dual}, the dual of $Q$ is $(n-k)$-translatable.

$(\Leftarrow)$: Since $Q$ is $k$-translatable then, by Theorem \ref{dual}, $[k(n-k)]_n=[-k^2]_n=1$ . This implies $[k^2]_n=n-1$, which, by Corollary \ref{C-alter}, shows that $Q$ is alterable. 
\end{proof}


\section*{\centerline{5. Translatable semigroups}}\setcounter{section}{5}\setcounter{theorem}{0}

Observe first that in a $k$-translatable groupoid $Q$ of order $n$ with a $k$-translatable sequence $a_1,a_2,\ldots,a_n$ we have
\begin{equation}\label{eas}
(x\cdot y)\cdot z=x\cdot (y\cdot z)\Longleftrightarrow a_{[k-ka_{[k-kx+y]_n}+z]_n}=a_{[k-kx+a_{[k-ky+z]_n}]_n}.
\end{equation} 
If $Q$ is left cancellative, then the right side of \eqref{eas} is equivalent to 
\begin{equation}\label{ee1}
[z-ka_{[k-kx+y]_n}]_n=[a_{[k-ky+z]_n}-kx]_n .
\end{equation}
Thus, a left cancellative, $k$-translatable groupoid $Q$ of order $n$ is a semigroup if and only if it satisfies \eqref{ee1}.

\begin{theorem}\label{semi}
A left cancellative, $k$-translatable groupoid $Q$ of order $n$ with a $k$-translatable sequence $a_1,a_2,\ldots,a_n$ is a semigroup if and only if $[k^2+k]_n=0$ and $a_i=[i-k-ka_k]_n$ for all $i=1,2,\ldots,n$.
\end{theorem}
\begin{proof}
$(\Rightarrow)$: Setting $x=y=n$ and $z=[i-k]_n$ in \eqref{ee1} we obtain $a_i=[i-k-ka_k]_n$.
Consequently, $a_k=[-ka_k]_n=[a_1+k-1]_n$, which implies
$[1-k-a_1]_n=[ka_k]_n=[ka_1+k^2-k]_n$. Hence $[1-k^2]_n=[a_1(k+1)]_n$. 
Since, setting $x=y=z=1$ in \eqref{ee1} we obtain $[a_1(k+1)]_n=[k+1]_n$, the last equation gives $[1-k^2]_n=[k+1]_n$. Hence, $[k^2+k]_n=0$.

$(\Leftarrow)$: From $a_i=[i-k-ka_k]_n$ we conclude $a_k=[-ka_k]_n$. Thus, $[-ka_k]_n=[k^2a_k]_n$. Now using the above and $[k^2+k]_n=0$ we obtain
$$
\arraycolsep=.5mm
\begin{array}{rll}
[z-ka_{[k-kx+y]_n}]_n&=&[z-k(y-kx-ka_k)]_n=[z-ky+k^2x+k^2a_k]_n\\[4pt]
&=&[z-ky-kx-ka_k]_n=[a_{[k-ky+z]_n}-kx]_n,
\end{array}
$$
which proves \eqref{ee1}. Hence $Q$ is a semigroup.
\end{proof}

\begin{corollary}\label{lmonoid}
A $k$-translatable semigroup is left cancellative if an only if it has a left neutral element.
\end{corollary}
\begin{proof}
If $Q$ is a left cancellative and $k$-translatable semigroup, then by previous theorem $a_i=[i-k-ka_k]_n$ and $[-ka_k]_n=a_k$. So, $a_{[-2a_k-1]_n}=[-k-1-a_k]_n$ and $a_{[-2a_k-1]_n}\cdot a_j=a_t$, where 
$$
t=[k-ka_{[-2a_k-1]_n}+a_j]_n=[k-k(-k-1-a_k)+(j-k-ka_k)]_n=j.
$$
Thus, $a_{[-2a_k-1]_n}$ is a left neutral element.

The converse statement is obvious.
\end{proof}

\begin{corollary}\label{C-unit}
A left unitary, $k$-translatable groupoid $Q$ of order $n$ is a semigroup if and only if $[k+k^2]_n=0$.
\end{corollary}

\begin{corollary}
A left unitary, $k$-translatable groupoid $Q$ of order $n=k+k^2$ is a semigroup with the multiplication given by the formula 
\begin{equation}\label{ee3}
(i+sk)\cdot (j+tk)=[j+(s+t-i+1)k]_n
\end{equation} 
where $i,j\in\{1,2,\ldots,k\}$ and $s,t\in\{0,1,2,\ldots,k\}$.
\end{corollary}
\begin{proof} 
Observe first that all elements of $Q$ can be written in the form $i+sk$, where $i\in\{1,2,\ldots,k\}$ and $s\in\{0,1,2,\ldots,k\}$.
Since $Q$ is left unitary and $k$-translatable, $x\cdot y=[k-kx+y]_n$ for all $x,y\in Q$. Therefore 
$
(i+sk)\cdot (j+tk)=[k-k(i+sk)+(j+tk)]_n=[j+(s+t-i+1)k]_n
$. It is straightforward to verify that this multiplication is associative.   
\end{proof}
Such defined semigroup is left cancellative and non-commutative. This semigroup is not a group. The equation $(x+yk)\cdot (j+tk)=(i+sk)$ always has $k$ solutions. All these solutions have the form $x+(k+1)m+yk$, $m=1,2,\ldots,k$.
\begin{corollary}
A groupoid $Q$ of order $n=k+k^2$ with the multiplication defined by $\eqref{ee3}$ is a left unitary, $k$-translatable semigroup.
\end{corollary}
\begin{proof} 
Indeed, this multiplication is associative and $1$ is its left neutral element. Also, $[(i+1)+sk]_n\cdot [j+(t+1)k]_n=[j+(t+s-i+1)k]_n= [i+sk]_n\cdot [j+tk]_n$, so $Q$ is $k$-translatable.
\end{proof}

\begin{theorem}\label{T-reord}
By re-ordering, any left cancellative $k$-translatable semigroup can be transformed into a left unitary $k$-translatable semigroup.
\end{theorem}
\begin{proof}
Let $Q$ be a left cancellative $k$-translatable semigroup and $a_1,a_2,\ldots,a_n$ be its $k$-translatable sequence with respect to the ordering $1,2,\ldots,n$. 
Consider a new ordering $b_1,b_2,\ldots,b_n$ of $Q$, where $b_s=[-a_k-k+s-2]_n$, $s=1,2,\ldots,n$.

Using Theorem \ref{semi}, it is not difficult to see that $b_i\cdot b_j=b_{[j-ki+k]_n}$. Thus $b_1$ is a left neutral element of $Q$ and, by Lemma \ref{L-basic} $(i)$, $Q$ is $k$-translatable with respect to the ordering $b_1,b_2,\ldots,b_n$.
\end{proof}

Thus by re-ordering and re-numeration of elements we can assume that each left cancellative $k$-translatable semigroup is left unitary with $1$ as its left neutral element.

As a consequence of Theorem \ref{T-reord} and Theorem \ref{T-izo} we have:

\begin{corollary}\label{C-izo}
Left cancellative $k$-translatable semigroups of the same order are isomorphic.
\end{corollary}

\begin{corollary}\label{cyclic}
A left cancellative groupoid $Q$ of order $n$ with an $(n-1)$-translatable sequence of the form $a_{[i+1]_n}=[a_i+1]_n$  is a cyclic group. 
\end{corollary}
\begin{proof}
An $(n-1)$-translatable groupoid is commutative. So, if it is left cancellative, then it also is right cancellative. Hence it is a quasigroup. Moreover, $k=n-1$ and $a_{[i+1]_n}=[a_i+1]_n$ imply  $[k+k^2]_n=0$ and $a_i=[i-k-ka_k]_n$ for all $i=1,2,\ldots,n$.  By Theorem \ref{semi} then, $Q$ is a semigroup. Hence $Q$ is a commutative group. By Corollary \ref{C-cyc} it is cyclic.
\end{proof}

\begin{corollary}\label{C-semi} A left cancellative $(n-1)$-translatable semigroup of order $n$ is isomorphic to the additive group $\mathbb{Z}_n$.
\end{corollary}
\begin{proof} By Theorem \ref{semi}, $a_i=[i+1+ak]_n$. Therefore, $a_{[i+1]_n}=[(i+1)+1+a_k]_n=[a_i+1]_n$. The result then follows from Corollary \ref{cyclic}.  
\end{proof}

\begin{proposition}
There are no idempotent $k$-translatable semigroups of order $n$.
\end{proposition}
\begin{proof}
An idempotent $k$-translatable semigroup of order $n$ is left cancellative. Thus, by Lemma \ref{L-cond} and Theorem \ref{semi}, in such semigroup for all $i\in Q$ should be
$i=a_{[k-ki+i]_n}=[i-ki-ka_k]_n$, i.e., $[k(i+a_k)]_n=0$, which is impossible.
\end{proof}

Denote by $E(Q)$ the set of all idempotents of a semigroup $Q$.

\begin{lemma}\label{idemp}
If $Q$ is a left cancellative, $k$-translatable semigroup of order $n$, then
\begin{enumerate}
\item[$(1)$] $E(Q)=\{i\in Q:[k(i+a_k)]_n=0\}$,
\item[$(2)$] $E(Q)$ is a subsemigroup of $Q$,
\item[$(3)$] $E(Q)$ equals the set of all left neutral elements of $Q$,
\item[$(4)$] $E(Q)$ is a right-zero semigroup.
\end{enumerate}
\end{lemma}
\begin{proof}
$(1)$: By Lemma \ref{L-cond} and Theorem \ref{semi}, we have

\medskip $i=i\cdot i\Leftrightarrow i=a_{[k-ki+i]_n}=[i-ki-ka_k]_n\Leftrightarrow [k(i+a_k)]_n=0$.

\medskip
$(2)$: If $i,j\in E(Q)$, then $(i\cdot j)\cdot (i\cdot j)=[j-ki-ka_k]_n\cdot [j-ki-ka_k]_n=[j-ki-ka_k-k(i+a_k)-k(j+a_k)]_n=[j-ki-ka_k]_n=i\cdot j$. So, $i\cdot j\in E(Q)$.

$(3)$: For every $i\in E(Q)$ and $j\in Q$ we have $i\cdot j=[j-ki-ka_k]_n=j$.

$(4)$: This follows from $(3)$.
\end{proof}

\begin{corollary}\label{reord}
If $1$ is a left neutral element of a left cancellative, $k$-translatable semigroup $Q$ of order $n$, then $E(Q)=\{j=[i+k(i-1)]_n: i\in Q\}$.
\end{corollary}

\begin{corollary} Right cancellative semigroups of order $n$ may be $k$-translatable only for $k=n-1$. Such semigroups are isomorphic to the additive groups $\mathbb{Z}_n$.  
\end{corollary}
\begin{proof} 
Let $Q$ be a right cancellative $k$-translatable semigroup of order $n$. Then the right cancellativity implies that each column of the multiplication table of $Q$ consists of $n$ distinct elements. Since $Q$ is $k$-translatable, the first row of its multiplication table has $n$ distinct elements, i.e., $Q$ has a left cancellable element. Thus, by Lemma \ref{L-lc}, it is left cancellative. So, $Q$ is an associative quasigroup. Hence $Q$ is a group. Corollary \ref{C-cyc} ends the proof.
\end{proof}
  
\begin{corollary} A left cancellative, $k$-translatable semigroup of order $n$ with $k$-translatable sequence $a_1,a_2,\ldots,a_n$ in which 
$a_k=1$ or $a_k=n-1$ is isomorphic to the group $\mathbb{Z}_n$ 
\end{corollary}
\begin{proof}
If $a_k=1$ then by Theorem \ref{semi}, $a_k=1=[-ka_k]_n=[-k]_n$. Hence $k=n-1$. If $a_k=-1$ then by Theorem \ref{semi} again, $a_k=n-1=[-ka_k]_n=k$. So, in both these cases the results follows from Corollary \ref{C-semi}.
\end{proof}

\begin{theorem}\label{decomp}
A left cancellative, $k$-translatable semigroup of order $n$ is a union of $t$ disjoint copies of cyclic groups of order $m$, where $m$ and $t$ are the smallest positive integers such that $[mk]_n=0$ and $[tm]_n=0$.
\end{theorem}
\begin{proof}
For every $i\in E(Q)$ consider the set $Q_i=Q\cdot i=\{x\cdot i:x\in Q\}$. Then, by associativity and Lemma \ref{idemp}, $(x\cdot i)\cdot (y\cdot i)=x\cdot (y\cdot i)=x\cdot [i-ky-ka_k]_n=[i-ky-kx-2ka_k]_n=y\cdot (x\cdot i)=(y\cdot i)\cdot (x\cdot i)$. So, $Q_i$ is a commutative semigroup contained in $Q$ and $i=i\cdot i$ is its neutral element. Since $(x\cdot i)\cdot (y\cdot i)=i$ for $y=[-x-2a_k]_n$, this semigroup is a commutative group. Moreover, for every $i\in E(Q)$, the mapping $\varphi(x\cdot i)=x\cdot 1$ is an isomorphism between $Q_i$ and $Q_1$. So, all groups $Q_i$ are isomorphic to $Q_1$. 
Since $Q$ is $k$-translatable, the group $Q_1$ contains elements of the form $(t+1)\cdot 1=[1-kt]_n$, $t=0,1,\ldots,m-1$, where $m$ is the smallest positive integer such that $[mk]_n=0$. It is not difficult to see that this group is generated by the element $2\cdot 1=[1-k]_n$.

Groups $Q_i$ are pairwise disjoint. If not, then for $z\in Q_i\cap Q_j$ we have $z=x\cdot i=y\cdot j$ for some $x,y\in Q$. Whence, multiplying by $i$ and $j$, and applying Lemma \ref{idemp}, we obtain $z\cdot i=x\cdot i=y\cdot i$ and $z\cdot j=x\cdot j=y\cdot j$. Thus $x\cdot i=y\cdot j=x\cdot j$, which by left cancellativity gives $i=j$. Hence $Q_i=Q_j$.

Obviously $\bigcup Q_i\subseteq Q$, $i\in E(Q)$. To prove the converse inclusion observe that, according to Theorem \ref{T-reord}, in $Q$ we can introduce a new ordering such that some element of $E(Q)$ can be identified with $1$. Then, by Corollary \ref{reord}, $[q+kq-k]_n\in E(Q)$ for each $q\in Q$. Thus, $q\cdot i=q$ for $i=[q+kq-k]_n$. Hence $Q\subseteq \bigcup Q_i$. So, $Q$ is a union of $t$ copies of $Q_1$, where $t$ is the number of idempotents of $Q$. Since $(1+wm)\cdot j=k-k(1+wm)+j=j$, the set $L=\{1+wm: w=1,2,\ldots\}$ consists of left identities. Since $t$ is the smallest positive integer such that $[tm]_n=0$, $L=\{1+wm: w=0,1,\ldots,t-1\}$. By Lemma \ref{idemp}, $L=E(Q)$.
\end{proof}

Note that $Q_i\cdot Q_j\subseteq Q_j$ for $i,j\in E(Q)$.

\begin{corollary}\label{ideals}
A left cancellative, $k$-translatable semigroup of order $n$ is a union of $t$ disjoint pairwise isomorphic left ideals of order $m$, where $m$ and $t$ are the smallest positive integers such that $[mk]_n=0$ and $[tm]_n=0$.
\end{corollary}

\begin{corollary}\label{c-decomp}
A left cancellative, $k$-translatable semigroup of order $n=k^2+k$ is a union of $k$ disjoint copies of cyclic groups isomorphic to $\mathbb{Z}_{k+1}$.
\end{corollary}

\begin{example}\rm
A left unitary $2$-translatable semigroup of order $6$ has two idempotents $1$ and $4$ and is a union of two cyclic groups isomorphic to the additive $\mathbb{Z}_3$. On the other hand, a
left unitary $3$-translatable semigroup of order $6$ has three idempotents $1,3$ and $5$ and is a union of three cyclic groups of order two.
\end{example}

\begin{theorem}
In left cancellative $k$-translatable semigroup every left and every right ideal is semiprime.
\end{theorem}
\begin{proof}
It is a consequence of our Theorem \ref{decomp} and Theorem 4.3 from \cite{Cliff}.
\end{proof}

Left cancellative $k$-translatable semigroups belong to many important classes of semigroups. Below we present a short survey of such classes. For definitions see for example \cite{Cliff}

\begin{theorem}
A left cancellative $k$-translatable semigroup of order $n$ is
\begin{enumerate}
\item[$(1)$] medial,
\item[$(2)$] anticommutative $($i.e., nowhere commutative$)$, if $(1+k,n)=1$,
\item[$(3)$] conditionally commutative,
\item[$(4)$] left commutative,
\item[$(5)$] left and right regular, 
\item[$(6)$] regular, 
\item[$(7)$] intra-regular,
\item[$(8)$] orthodox,
\item[$(9)$] a Clifford right semigroup, 
\item[$(10)$] a Clifford left semigroup, if $(k,n)=1$.
\end{enumerate}

\end{theorem}
\begin{proof}
$(1)$: This is a consequence of Theorem \ref{med}. 

$(2)$: Indeed, $i\cdot j=j\cdot i$ means that $[j-ki-ka_k]_n=[i-kj-ka_k]_n$, i.e., $[(j-i)(1+k)]_n=0$, which for $(1+k,n)=1$ gives $i=j$.

$(3)$: It is easy to see that $i\cdot j=j\cdot i$ implies $i\cdot x\cdot j=j\cdot x\cdot i$ for all $x\in Q$. 

$(4)$: Since $Q$ has a left neutral element, $(1)$ implies $i\cdot j\cdot x=j\cdot i\cdot x$ for all $i,j,x\in Q$.

$(5)$: $x\cdot (j\cdot j)=j$ and $(j\cdot j)\cdot y=j$ for each $j\in Q$ and $x=[-j-2ka_k]_n$, $y=[j+2kj+2ka_k]_n$.

$(6)$: $i=i\cdot x\cdot i$ for every $i\in Q$ and $x=[-i-2a_k]_n$. 

$(7)$: $i=x\cdot i\cdot i\cdot y$ for every $i\in Q$ and $x=[-2i-3a_k]_n$, $y=i$. 

$(8)$: It is a consequence of Lemma \ref{idemp} and $(6)$.

$(9)$: $Q\cdot i\subseteq i\cdot Q$ for every $i\in Q$, because $j\cdot i=i\cdot s$ for $s=[i+k(i-j)]_n$.  

$(10)$: Analogously, $i\cdot Q\subseteq Q\cdot i$ for every $i\in Q$, because $i\cdot j=s\cdot i$ for $s=[t(i-j)+i]_n$, where $[tk]_n=1$. Such $t$ exists only for $(k,n)=1$.
\end{proof}

\begin{theorem}
A paramedial left cancellative $k$-translatable semigroup is a cyclic group.
\end{theorem}
\begin{proof}
Comparing Theorems \ref{para} and \ref{semi} we can see that such semigroup is $(n-1)$-translatable and $a_{[i+1]_n}=[(i+1)+1+a_k]_n=[a_i+1]_n$. Corollary \ref{cyclic} completes the proof.
\end{proof}

Below we present several results on $k$-translatable semigroups which are not left cancellative.

\begin{theorem}\label{Tsem}
If the first row of the multiplication table of a naturally ordered $k$-translatable groupoid $Q=\{1,2,\ldots,n\}$ has the form $a_1,a_2,\ldots,a_n$, then $Q$ is a semigroup with the multiplication $i\cdot j=a_j$ if and only if $a_s=a_{[s+k]_n}=a_{a_s}$ for any $s\in Q$. 
\end{theorem}
\begin{proof}
$(\Rightarrow)$: Indeed, $(i\cdot j)\cdot s=a_s$ and $i\cdot (j\cdot s)=a_{a_s}$ imply $a_s=a_{a_s}$ for all $s\in Q$.
From $k$-translability we obtain $a_s=a_{[s+k]_n}$.

$(\Leftarrow)$:
First note that  $a_s=a_{[s+k]_n}$ implies $a_{[s-k]_n}=a_{[s-k+k]_n}=a_s$. Thus, for any, $i,j,s\in Q$ we have
$$
a_{[k-ka_{[k-ki+j]_n}+s]_n}=a_{[k+s]_n}=a_s=a_{a_s}=a_{a_{[k-ki+s]_n}}=a_{[k-ki+a_{[k-kj+s]_n}]_n},
$$
which, by \eqref{eas}, means that $Q$ is a semigroup. The rest is a consequence of Lemma \ref{L-basic} $(i)$.
\end{proof}
  
Note by the way that in such defined semigroup the multiplication takes no more than $d=(k,n)$ values. So, in the case $(k,n)=1$ it is semigroup with  multiplication having one constant value.

Theorem \ref{Tsem} together with Lemma \ref{L-3} give the possibility to construct $k$-translatable semigroups which are not left cancellative.

\begin{example} Let $Q=\{1,2,\ldots,n\}$ be naturally ordered set. Then
\begin{enumerate}
\item[$(1)$] the sequences $(12,12,9,10,9,12,12,12,9,10,9,12)$, $(1,1,1,10,11,10,1,1,1,10,11,10)$ and $(5,8,8,8,5,6,5,8,8,8,5,6)$ define three isomorphic $6$-translatable semigroups of order $n=12$,

\item[$(2)$] the sequences $(1,3,3,1,3,3)$, $(2,2,6,2,2,6)$ and $(1,5,1,1,5,1)$ yield isomorphic $3$-translatable semigroups of order $6$.

\item[$(3)$] the sequences $(1,5,4,4,5,1,5,4,4,5)$, $(4,3,3,4,10,4,3,3,4,10)$ and 

$(1,2,8,2,1,1,2,8,2,1)$ define isomorphic $5$-translatable semigroups of order $n=10$.
\end{enumerate}
\end{example}

\begin{corollary}
A naturally ordered semigroup $Q=\{1,2,\ldots,n\}$ with a $k$-trans\-latable sequence $a_1,a_2,\ldots,a_n$ in which 
\begin{enumerate}
\item[$(1)$] $a_1=a_2=n$ or 
\item[$(2)$] $a_1=1$ and $a_j=j+1$ for some $j\ne 1$ 
\end{enumerate}
has multiplication defined by the formula $i\cdot j=a_j$. 
\end{corollary}
\begin{proof}
Let $Q$ be a $k$-translatable semigroup. Then, by \eqref{eas}, for $x=1$, $y=j$ and $z=s$ we have
\begin{equation}\label{tsem}
a_{[k-ka_j+s]_n}=a_{a_{[k-kj+s]_n}},
\end{equation}
which for $j=1$ gives $a_{[k-ka_1+s]_n}=a_{a_s}$. From this, for $a_1=n$, we obtain $a_{[s+k]_n}=a_{a_s}$, whence we get $a_{s}=a_{a_{[s-k]_n}}$.
Thus, 
$$
a_s=a_{a_{[s-k]_n}}=a_{a_{[k-k2+s]_n}}\stackrel{\eqref{tsem}}{=}a_{[k-ka_2+s]_n}=a_{[k+s]_n}
$$
for $a_2=n$. So, for $a_1=a_2=n$ the conditions of Theorem \ref{Tsem} are satisfied. Hence \ $i\cdot j=a_j$ for all $i,j\in Q$.

Now if $a_1=1$ and $a_j=j+1$ for some $j\ne 1$, then $1\cdot 1=a_1=1$ and $a_s=1\cdot s=(1\cdot 1)\cdot s=1\cdot (1\cdot s)=1\cdot a_s=a_{a_s}$. This together with \eqref{tsem} implies $a_{[k-kj+s]_n}=a_{a_{[k-kj+s]_n}}=a_{[k-ka_j+s]_n}$, which for $s=[ka_j-k+t]_n$ gives $a_{[k(a_j-j)+s]_n}=a_t$. But $a_j-j=1$ for some $j\ne 1$, so $a_{[k+t]_n}=a_t$. This means that also in this case the condition of Theorem \ref{Tsem} are satisfied. Thus, also in this case $i\cdot j=a_j$ for all $i,j\in Q$.
\end{proof}

\begin{corollary}
A $k$-translatable groupoid $Q=\{1,2,\ldots,n\}$ with an idempotent element $j\in Q$ such that $j=j\cdot [j-1]_n$ is a semigroup if and only if for all $s\in Q$ 
\begin{enumerate}
\item[$(i)$] $j\cdot[j-1+s-k]_n=j\cdot [j-1+s]_n=j\cdot [j-1+s+k]_n$ and 
\item[$(ii)$] $j\cdot[j-1+s]_n=j\cdot [(j-1)+(j\cdot [j-1+s]_n)]_n$. 
\end{enumerate}
\end{corollary}
\begin{proof}
First we introduce on $Q$ a new ordering $b_1,b_2,\ldots,b_n$ starting with an idempotent $j$ and such that $b_s=[j-1+s]_n$ for all $s\in Q$.
Then, by repeated application of Lemma \ref{L-3}, we can see that $a_1,a_2,\ldots,a_n$, where $a_s=j\cdot b_s$, is a $k$-translatable sequence with respect to this new ordering.

$(\Rightarrow)$: Since $Q$ is a $k$-translatable semigroup, for any $s\in Q$ we have
$$\arraycolsep=.5mm
\begin{array}{rl}
a_s&=j\cdot b_s=(j\cdot [j-1]_n)\cdot b_s=(b_1\cdot b_n)\cdot b_s=b_1\cdot (b_n\cdot b_s)\\[2pt]
&=j\cdot ([j-1]_n\cdot [j-1+s]_n)=j\cdot (j\cdot [j-1+s+k]_n)\\[2pt]
&=(j\cdot j)\cdot [j-1+s+k]_n=j\cdot [j-1+s+k]_n\\[2pt]
&=j\cdot b_{s+k}=a_{s+k}.
\end{array}
$$
From this, setting $s=[t-k]_n$ we obtain $a_t=a_{[t-k]_n}$, and so $a_s=a_{[s+k]_n}=a_{[s-k]_n}$. This proves $(i)$ and, together with \eqref{eas}, implies $a_s=a_{a_s}$. The last is equivalent to $(ii)$.

$(\Leftarrow)$: $(i)$ and $(ii)$ imply $a_s=a_{[s+k]_n}=a_{[s-k]_n}=a_{a_s}$ for all $s\in Q$. Therefore for $i,j,s\in Q$ we have $a_{[k-ka_{[k-ki+j]_n}+s]_n}=a_{[k-ki+a_{[k-kj+s]_n}]_n}$ and $Q$ is a semigroup.
\end{proof}

For $j=1$ from the above Corollary we obtain
\begin{corollary}\label{nsemi}
A naturally ordered, $k$-translatable groupoid $Q=\{1,2,\ldots,n\}$ with an idempotent element $1$ such that $1=1\cdot n$ is a semigroup if and only if for all $s\in Q$ we have 
$1\cdot [s-k]_n=1\cdot s=1\cdot [s+k]_n=1\cdot (1\cdot s)$. 
\end{corollary}

\section{Embeddings of left unitary translatable semigroups}

Suppose that $Q_1,Q_2,\ldots,Q_t$ are disjoint copies of the same left unitary $k$-translata\-ble semigroup of order $n$. Is there a product $*$ on $Q=\bigcup Q_i$ such that $(Q,*)$ is a left unitary translatable semigroup, with every $(Q_i,*_i)$ isomrphic to $(Q_i,*)$?

We find two different conditions under which this is possible. In each case $k=tq$, for some positive integer $q$. If $[k+k^2]_{tn}=0$ then we can find a product $*$ on $Q$ such that $(Q,*)$ is a left unitary $k$-translatable semigroup, with every $(Q_i,*_i)$ isomorphic to $(Q_i,*)$. If $n=k+k^2$ then we can construct a left unitary $(k+(t-1)n)$-translatable semigroup $(Q,*)$, such that every $(Q_i,*_i)$ is isomorphic to $(Q_i,*)$. 

To prove these results, we need the following Lemma. The proof is straightforward and is omitted.

\begin{lemma}\label{divid}
Let $t$ and $x$ be positive integers. Then 
\begin{enumerate}
\item[$(a)$] \ $[tx]_{tn}=t[x]_n$,
\item[$(b)$] \ $[t(x-1)]_{tn}=t[[x]_n-1]_n$,
\item[$(c)$] \ if $n=k+k^2$ and $k=tq$, then $[(k+(t-1)n)+(k+(t-1)n)^2]_{tn}=0$.
\end{enumerate}
\end{lemma}

\begin{theorem}\label{T-62} 
Let $k=tq$ and $[k+k^2]_{tn}=0$. Suppose that $(Q_i,*_i)$, $i=1,2,\ldots,t$, are pairwise disjoint left unitary $k$-translatable semigroup with ordering $i_1,i_2,\ldots,i_n$ and the cardinality $n$. Then $Q=\bigcup Q_i$, with ordering defined by $i_r=t(r-1)+i$\, $(i=1,2,\ldots,n \ and \ r=1,2,\ldots,n)$, is a left unitary $k$-translatable semigroup, with respect to the product $*$ defined by 
\begin{equation}\label{EE}
i_r*j_s=j_x, \ \ {where} \ \ x=[k-kr+q-qi+s]_n 
\end{equation} 
in which each $(Q_i,*)$ isomorphic to $(Q_i,*_i)$. 

Furthermore, if we define a product $*$ on $Q=\bigcup Q_i$ by \eqref{EE}, then there is an ordering on $Q$ such that $(Q,*)$ is a left unitary $k$-translatable semigroup in which each $(Q_i,*)$ isomorphic to $(Q_i,*_i)$. 
\end{theorem}
\begin{proof} Note that $[k+k^2]_{tn}=0$ implies $[k+k^2]_n=0$. Consider on $Q=\bigcup Q_i$ the ordering $i_r=t(r-1)+i$ and define the first row of the multiplication table of $*$ as $a_1,a_2,\ldots,a_{tn}$, where $a_i=i$ for all $i=1,2,\ldots,tn$. We make the sequence $a_1,a_2,\ldots,a_{tn}$ into a $k$-translatable sequence in a natural way, as follows. The last $k$ entries of the first row become the first $k$ entries of the second row and the first $(tn-k)$ entries of the first row become the last $(tn-k)$ entries of the second row. Continue in this manner, making the last $k$ entries of the $i$-th row become the first $k$ entries of the $(i+1)$-th row and the first $(tn-k)$ entries of the $i$-th row become the last $(tn-k)$ entries of the $(i+1)$-th row. By definition, the resultant groupoid $(Q,*)$ is $k$-translatable, with $1_1=1$ as a left identity element. Since $[k+k^2]_{tn}=0$, by Corollary \ref{C-unit},  $(Q,*)$ is a left unitary $k$-translatable semigroup, with product $r*s=a_{[k-kr+s]_{tn}}=[k-kr+s]_{tn}$. Therefore, for all $i_r,j_s\in Q$ we obtain
$$
\arraycolsep=.5mm\begin{array}{rl}
i_r*j_s&=[t(r-1)+i]_{tn}*[t(s-1)+j]_{tn}=[t(q-kr+k-qi+s-1)+j]_{tn}\\[3pt]
&=[t([q-kr+k-qi+s]_n-1)+j]_{tn}=j_w ,
\end{array}
$$
where $w=[q-kr+k-qi+s]_n$. Since, by Lemma \ref{divid}, $w=j[(q-kr+k-qi+s)]_n$, the above proves \eqref{EE} and shows that each $Q_j$ is a left ideal of $(Q,*)$.

Now, we define on $Q_i$ new ordering by putting $i_s=s$ for all $s=1,2,\ldots,n$. Then $r*s=i_r*i_s=i_x$, where $x=[k-kr+q-qi+s]_n=[k-kr+q-q(i+1)+s+q]_n$. So, $r*s=i_{[r+1]_n}*i_{[s+k]_n}=[r+1]_n*[s+k]_n$. By Lemma \ref{L-basic} $(ii)$, $(Q_i,*)$ is $k$-translatable.

Since $0=[k+k^2]_{tn}=[t(q+tq^2)]_{tn}$ by Lemma \ref{divid} $(a)$, $[q+tq^2]_n=0$. Therefore, $[tq^2]_n=[-q]_n$. If $r=1-q(1-i)$, then $[k-kr]_n=[k(1-r)]_n=[tq^2(1-i)]_n=[-q(1-i)]_n$. By \eqref{EE} therefore, for $r=1-q(1-i)$ we have $i_r*i_s=i_s$. Thus $i_r$ is a left neutral element of $(Q_i,*)$.

By Lemma \ref{L-3}, we can re-order $(Q_i,*)$ such that by Corollary \ref{C-unit}, $(Q_i,*)$ is a left unitary $k$-translatable groupoid of order $n$. Then, since $[k+k^2]_n=0$, by Corollary \ref{C-unit}, is a left unitary $k$-translatable semigroup. By Corollary \ref{C-izo} then, $(Q_i,*)$  is isomorphic to $(Q_i,*_i)$. This proves the first part of the Theorem.

To prove the second part we introduce on $Q=\bigcup Q_i$ an ordering $i_s=t(s-1)+i$ and consider a product $*$ defined by \eqref{EE}. 
Then, 
$$\arraycolsep=.5mm
\begin{array}{rl}
[i_r+1]_{tn}*[j_s+k]_{tn}&=[t(r-1)+i+1]_{tn}*[t(s-1)+j+k]_{tn}\\[4pt]
&=[t(r-1)+(i+1)]_{tn}*[t(s+q-1)+j]_{tn}\\[4pt]
&=(i+1)_r*j_{s+q}=j_x ,
\end{array}
$$
where $x=[(k-kr+q-q(i+1)+s+q)]_n=[(k-kr+q-qi+s)]_n$. So, $[i_r+1]_{tn}*[j_s+k]_{tn}=j_x=i_r*j_s$.
Thus $(Q,*)$ is $k$-translatable. Since, by \eqref{EE}, $1_1=1$ is a left neutral element of $(Q,*)$ and $[k+k^2]_{tn}=0$, by Corollary \ref{C-unit} then, $(Q,*)$ is a left unitary $k$-translatable semigroup.

Analogously we can see that each $(Q_i,*)$ with natural ordering $s=i_s$ is a left unitary $k$-translatable semigroup. By Corollary \ref{C-izo}, semigroups $(Q_i,*)$ and $(Q_i,*_i)$ are isomorphic. This proves the second part of the Theorem.  
\end{proof}

\begin{theorem}\label{T-63} 
Let $k=tq$ and $n=k+k^2$. Suppose that $(Q_i,*_i)$, $i=1,2,\ldots,t$, are pairwise disjoint left unitary $k$-translatable semigroups with ordering $i_1,i_2,\ldots,i_n$ and the cardinality $n$. Then $Q=\bigcup Q_i$ with the ordering $i_r=t(r-1)+i$ is a left unitary $(k+(t-1)n)$-translatable semigroup with respect to the product $*$ defined by 
\begin{equation}\label{EEE}
i_r*j_s=j_x, \ \ {where} \ \ x=[k-kr+kq(i-1)+s]_n 
\end{equation} 
in which each $(Q_i,*)$ isomorphic to $(Q_i,*_i)$. 

Furthermore, if on $Q=\bigcup Q_i$ a product $*$ is defined by \eqref{EEE}, then $Q$ can be ordered in such a way that $(Q,*)$ becomes a left unitary $(k+(t-1)n)$-translatable semigroup in which each $(Q_i,*)$ is isomorphic to $(Q_i,*_i)$. 
\end{theorem}
\begin{proof} In a similar way as in the proof of Theorem \ref{T-62} we can construct a multiplication table for $(Q,*)$ in this way that $(Q,*)$  becomes a left unitary $(k+(t-1)n)$-translatable semigroup. Then, according to the $(k+(t-1)n)$-translatability, Lemma \ref{divid} and $k+(t-1)n=tk+(t-1)k^2$, we have
$$
\arraycolsep=.5mm\begin{array}{rl}
i_r*j_s&=[t(r-1)+i]_{tn}*[t(s-1)+j]_{tn}\\[3pt]
&=[(k+(t-1)n)-(k+(t-1)n)(t(r-1)+i)+(t(s-1)+j)]_{tn}\\[3pt]
&=[(tk+(t-1)t^2q^2)-(tk+(t-1)t^2q^2)(t(r-1)+i)+(t(s-1)+j)]_{tn}\\[3pt]
&=[t\{(k+(t-1)tq^2)-(k+(t-1)tq^2)(t(r-1)+i)+s-1\}+j]_{tn}=j_x ,
\end{array}
$$
where $$
\arraycolsep=.5mm\begin{array}{rl}
x&=[(k+(t-1)tq^2)-(k+(t-1)tq^2)(t(r-1)+i)+s]_n\\[3pt]
&=[-tq^2+tq^2(t(r-1)+i)+s]_n=[rk^2-k^2+kq(i-1)+s]_n\\[3pt]
&=[k-kr+kq(i-1)+s]_n
\end{array}
$$
because $k+k^2=n$.

This proves \eqref{EEE} and shows that each $Q_j$ is a left ideal of $(Q,*)$. Next, similarly as in the previous proof, we can prove that each semigroup $(Q_i,*)$ has a left neutral element $i_r$, where $r=q(i-1)+1$, and is isomorphic to $(Q_i,*_i)$. This completes the proof of the first part of Theorem.

To prove the second part we introduce on $Q=\bigcup Q_i$ an ordering $i_s=t(s-1)+i$ and consider the product defined \eqref{EEE}. Then, 
$$\arraycolsep=.5mm
\begin{array}{rl}
[i_r+1]_{tn}*[j_s+k+(t-1)n]_{tn}&=[t(r-1)+i+1]_{tn}*[t(s-1)+j+k+(t-1)n]_{tn}\\[4pt]
&=[t(r-1)+(i+1)]_{tn}*[t(s-kq-1)+j]_{tn}\\[4pt]
&=(i+1)_r*j_{s-kq}=j_x ,
\end{array}
$$
where $x=[(k-kr+(s-kq)+kqi)]_n=[(k-kr+s-kq(i-1))]_n$. 

So, $[i_r+1]_{tn}*[j_s+k+(t-1)n]_{tn}=j_x=i_r*j_s$,
which means that $(Q,*)$ is a $(k+(t-1)n)$-translatable semigroup with a left neutral element $1_1=1$. It is not difficult to see that each $(Q_i,*)$ with the ordering $i_s=s$ is a $k$-translatable semigroup isomorphic to $(Q_i,*)$.
\end{proof}

\begin{example} Suppose that $Q$ is a left unitary $8$-translatable semigroup of order $n=12$. Since $8=2\times 4$ and $[8+8^2]_{24}=0$, there is a left unitary $8$-translatable semigroup $G$ of order $24$ that is a disjoint union of two copies of $Q$. It is a disjoint union of $8$ copies of $\mathbb{Z}_3$.
\end{example}
\begin{example} Suppose that $Q$ is a left unitary $8$-translatable semigroup of order $n=72$. Then there is a left unitary $224$-translatable semigroup of order $288$ that is a disjoint union of $4$ copies of $Q$. There is also a left unitary $512$-translatable semigroup $G$ of order $576$ that is a disjoint union of $8$ copies of $Q$. Also it is a disjoint union of $64$ copies of $\mathbb{Z}_9$.
\end{example}

Putting in the last theorem $t=2$ we obtain

\begin{corollary}\label{dec-semi} Let $(Q,\cdot)$ and $(G,\circ)$ be left unitary $k$-translatable semigroups of order $n=k+k^2$, with $k=2q$, $Q=\{1,2,\ldots,n\}$, $G=\{g_1,g_2,\ldots,g_n\}$ and $Q\cap G=\emptyset$. Then $B=Q\cup G$ can be transformed into a left unitary $(k+n)$-translatable semigroup with product $\ast$ such that
\begin{enumerate}
  \item[$(i)$]  $(Q,\cdot)=(Q,\ast)$,
  \item[$(ii)$]  $(G,\circ)$ is isomorphic to $(G,\ast)$,
  \item[$(iii)$]  $g_i\ast g_j=g_{[k(1+q)-ki+j]_n}$,
  \item[$(iv)$]  $i\ast g_j=g_{[k-ki+j]_n}$ for all $i,j\in Q$ and 
  \item[$(v)$]  $g_i\ast j=[k(1+q)-ki+j]_n$ for all $i,j\in Q$.
\end{enumerate}

Conversely, if $n=k+k^2$ and $B=Q\cup G$ with the ordering $1,g_1,2,g_2,\ldots,n,g_n$ has the product $\ast$ defined by $(i)$, $(iii)$, $(iv)$ and $(v)$ then $(B,\ast)$ is a left unitary $(k+n)$-translatable semigroup satisfying $(ii)$. Furthermore, $Q\ast G=G$ and $G\ast Q=Q$.
\end{corollary}


\noindent
W.A. Dudek \\
 Faculty of Pure and Applied Mathematics,
 Wroclaw University of Science and Technology,
 50-370 Wroclaw,  Poland \\
 Email: wieslaw.dudek@pwr.edu.pl\\[4pt]
R.A.R. Monzo\\
Flat 10, Albert Mansions, Crouch Hill, London N8 9RE, United Kingdom\\
E-mail: bobmonzo@talktalk.net

\end{document}